\documentclass[12pt,oneside]{amsart}

\usepackage[
    margin=1in, 
    hmarginratio=1:1, 
    vmarginratio=1:1
]{geometry} 

\usepackage[T1]{fontenc}
\usepackage[utf8]{inputenc}

\usepackage{amsmath}
\usepackage{amssymb}

\usepackage{enumitem}

\usepackage{graphicx}

\usepackage[colorlinks]{hyperref}

\newtheorem{thm}{Theorem}[section]

\newtheorem{claim}[thm]{Claim}

\newtheorem{cor}[thm]{Corollary}

\newtheorem{lem}[thm]{Lemma}
\newtheorem{prop}[thm]{Proposition}

\newtheorem{fact}[thm]{Fact}

\def\L{\mathcal{L}}

\DeclareMathOperator{\Span}{Span}

\DeclareMathOperator{\gr}{gr}

\def\Ker{{\operatorname{Ker}}}

\def\wt{{\operatorname{wt}}}
\def\idd{{\operatorname{id}}}

\def\End{{\operatorname{End}}}
\def\her{{\operatorname{her}}}

\title[Growth of Lie algebras and Jordan algebras]{On the growth of Lie algebras and Jordan algebras}

\author{Be'eri Greenfeld} \address{Department of Mathematics and Statistics, CUNY Hunter College, 695 Park Avenue, New York NY 10065, USA} \email{beeri.greenfeld@hunter.cuny.edu}

\begin{document}

\maketitle

\begin{abstract}
We prove that every increasing, polynomially bounded function is realizable, up to bi-Lipschitz equivalence, as the growth function of some finitely generated Lie algebra. This does not extend to general increasing subexponential functions. We further prove that the classes of growth functions of Jordan algebras and associative algebras coincide, answering in the affirmative a question of Mart\'inez and Zelmanov.
\end{abstract}

\section{Introduction}

Let $\Bbbk$ be a field of characteristic different from two and let $A$ be a finitely generated (not necessarily associative) $\Bbbk$-algebra. If $A$ is infinite-dimensional as a $\Bbbk$-vector space, a useful invariant quantifying `how fast it becomes infinite-dimensional' is given by its \emph{growth function}. Let $X$ be a finite generating set of $A$, and let $X^{\leq n}$ be the set of all products of at most $n$ elements from $X$. The growth function of $A$ with respect to $X$ is 
\[
\gamma_{A,X}(n) = \dim_\Bbbk \Span_\Bbbk X^{\leq n}.
\]
This function depends numerically on the choice of $X$, but, given any two finite generating sets $X,Y$ of $A$, the resulting growth functions are asymptotically equivalent. We say that two non-decreasing functions $f,g\colon \mathbb{N}\rightarrow \mathbb{N}$ are bi-Lipschitz equivalent, denoted $f\sim g$, if for some constants $C,D>0$ we have $f(n)\leq Cg(Dn)$ and $g(n)\leq Cf(Dn)$ for all $n$. We write $f\preceq g$ if $f(n)\leq Cg(Dn)$ for all $n$. Then $\gamma_{A,X}(n)\sim \gamma_{A,Y}(n)$.

\medskip

The study of the growth functions of infinite-dimensional algebras is parallel to the study of group growth, one of the most important asymptotic invariants in geometric group theory. Given a finitely generated group $G$ and a finite generating set $X$, the growth function of $G$ with respect to $X$ is given by $\gamma_{G,X}(n)=|(X\cup X^{-1})^{\leq n}|$. Geometrically, this is the volume of the ball of radius $n$ in the Cayley graph of $G$ with respect to $X$, centered at the identity.
By Gromov's celebrated theorem \cite {Gromov}, groups of polynomially bounded growth are virtually nilpotent. Grigorchuk \cite{Grigorchuk} constructed the first groups of intermediate growth, that is, not bounded by any polynomial but slower than every exponential function; see also \cite{Bartholdi,ErschlerZheng}. Since then, groups with a variety of interesting intermediate growth functions have been constructed; see e.g. \cite{BartholdiErschler,KassabovPak,NekrashevychSimpleGroups}. 
Nevertheless, the \emph{inverse problem of growth functions of groups}, that is, characterizing the possible growth functions of finitely generated groups, remains wide open. No groups of intermediate growth slower than that of Grigorchuk's group are known, and Grigorchuk's gap conjecture asserts that no growth functions of groups can be faster than every polynomial and slower than $\exp(\sqrt{n})$. See \cite{ShalomTao} for a partial gap result in this direction.
The growth function of every group coincides with the growth of its group algebra. Interesting asymptotic invariants of groups are reflected by the growth of associated Lie algebras, providing group theorists with a rich Lie-theoretic toolkit; for instance, Grigorchuk was able to prove his gap conjecture for residually-$p$ groups; see \cite{BartholdiGrigorchuk,GrigorchukResiduallyp}.

\medskip

The study of the growth functions of associative algebras originated in the work of Gel'fand and Kirillov on geometric representation theory of algebraic groups, where growth was used to define a birational invariant for universal enveloping algebras \cite{GelfandKirillov}. The class of growth functions of associative algebras coincides with the class of growth functions of semigroups, thus properly containing the class of growth functions of groups. The richness, or smoothness, of this space makes the \emph{inverse problem of growth functions of associative algebras} more accessible. Growth functions of associative algebras are either bounded, linear, or at least quadratic; the latter is Bergman's gap theorem \cite{Bergman}. Over the years, many interesting contributions have been made toward solving this inverse problem; see e.g. \cite{AAJZ,BelovBorisenkoLatyshev,BorhoKraft,GreenfeldSimple,GreenfeldPrimePrimitive,GaApp,NekrashevychAlgebras,Shneerson,SmoktunowiczBartholdi,Trofimov,Vishne,Warfield} and references therein. The inverse problem was eventually resolved by Bell and Zelmanov \cite{BZ}, who remarkably gave a complete characterization of the growth functions of associative algebras. Essentially, they proved that the only restrictions on such growth functions come from Bergman's gap theorem and a submultiplicativity-type condition on the discrete derivative.

\medskip

The \emph{inverse problem of growth functions of Lie algebras} remains mysterious. The class of growth functions of Lie algebras is richer than the class of growth functions of associative algebras. Given any associative algebra, one can construct a Lie algebra with an equivalent growth function \cite{AlahmadiAlharthi,AAJZ_Lie}, but growth functions of Lie algebras need not be submultiplicative, nor do they obey Bergman's linear-to-quadratic gap theorem. The only `natural' conditions are that these functions be increasing (unless bounded) and exponentially bounded. Additional interesting examples of growth functions of Lie algebras have been constructed in \cite{PetrogradskyOscillating,ShestakovZelmanov}.

In the polynomial realm, we give a full solution of the inverse problem for growth functions of Lie algebras:

\begin{thm} \label{thm:Lie}
Let $\Bbbk$ be a field. Let $f\colon \mathbb{N} \rightarrow \mathbb{N}$ be an increasing, polynomially bounded function. Then there exists a finitely generated Lie algebra $L$ over $\Bbbk$ such that $\gamma_L(n)\sim f(n)$.
\end{thm}

This characterization does not extend to general increasing subexponential functions (Proposition \ref{prop:Liegap}). We prove an analog of Theorem \ref{thm:Lie} for associative algebras (Theorem \ref{thm:associative}), and derive a corresponding result for Jordan algebras (Corollary \ref{cor:Jordan}).

\medskip

Jordan algebras are an important class of non-associative algebras. They originated in physics, since Hermitian operators form a Jordan algebra under a natural operation $x\circ y = \frac{1}{2}(xy+yx)$. This construction allows one to equip every associative algebra $A$ with an associated Jordan structure, denoted $A^{(+)}$. This is analogous to how every associative algebra becomes a Lie algebra under the Lie bracket $[x,y]=xy-yx$. 
By the Poincar\'e--Birkhoff--Witt theorem, every Lie algebra embeds in a Lie algebra arising from an associative algebra in this way. In sharp contrast, not every Jordan algebra embeds in an algebra of the form $A^{(+)}$. This leads to the distinction between special Jordan algebras (those embeddable in some $A^{(+)}$) and exceptional Jordan algebras (those not embeddable in this manner). Jordan algebras play an important role in the theory of Lie algebras and, in particular, they played a central role in Zelmanov's solution of the Restricted Burnside Problem \cite{Jacobson,McCrimmon,ZelmanovRBPodd,ZelmanovRBP2,ZelmanovJCA}. 
Jordan algebras, and more so, Jordan superalgebras, also arise naturally in the algebraic study of superconformal symmetry. In particular, certain superconformal Lie
algebras correspond, via the Tits--Kantor--Koecher construction, to
graded-simple Jordan superalgebras of linear growth. The corresponding
Jordan superalgebras were classified in characteristic zero in the
monumental work of Kac, Mart\'inez and Zelmanov
\cite{KacMartinezZelmanov}. See also \cite{KacVanDeLeur}.

The growth functions of Jordan algebras have therefore naturally attracted attention over the years. Mart\'inez and Zelmanov \cite{Martinez,MZ22} proved that special Jordan algebras have growth functions equivalent to those of their associative envelopes. Interestingly, they obtained a similar result at the other extreme, for Jordan algebras that are `very far' from being special \cite[Theorem~4]{MZ22}. They also proved that arbitrary Jordan algebras satisfy an analog of Bergman's gap theorem \cite{MZ96} (this does not carry over to Jordan superalgebras \cite{PetrogradskyShestakov}). The question of whether the classes of growth functions of Jordan algebras and associative algebras coincide was left open \cite[Question~1]{MZ22}. The next result answers it in the affirmative, thereby resolving the \emph{inverse problem for growth functions of Jordan algebras}:

\begin{thm} \label{thm:Jordan}
The following classes coincide, up to bi-Lipschitz equivalence:
\begin{itemize}
    \item Growth functions of finitely generated Jordan algebras
    \item Growth functions of finitely generated associative algebras
    \item Bounded functions, linear functions, and increasing, submultiplicative functions $f\colon \mathbb{N}\rightarrow~\mathbb{N}$ satisfying $f(n)\succeq~n^2$.
\end{itemize}
\end{thm}

\subsection*{Methodology and AI assistance}
AI assistance (ChatGPT 5.6 Sol and 6.0), brought below in detail, has been used to help develop certain combinatorial observations, constructions, and proof arguments and, importantly, to propose discrete versions of concepts discussed in various papers in analysis, which turned out to be useful in our proof of Theorem \ref{thm:Jordan}.

During the discussion on the proof strategy for Theorem \ref{thm:Lie}, AI pointed out the simple but useful observation that one loses `one degree of $N$' when considering Lie monomials with a restricted number of occurrences of one of the letters, which allows us to `work in strips' (as in Equations \eqref{eq:for intro},\eqref{eq:for intro2} therein). The idea to `pump' a module with prescribed growth, using an action of a Lie algebra coming from an associative algebra of polynomial growth is by the author, and its manifestation through the framework of a semidirect product was suggested by AI.

In the proof of Theorem \ref{thm:Jordan}, AI was used to recall some standard facts about Jordan algebras (in Step I); to suggest using a discrete version of an analytic notion from \cite{GrochenigWeights} used in the modification $\hat{f}$ in Equation \eqref{eq:Jordan4.5}; to suggest using a multiplicative version of the definition given in \cite{GoswamiPales} in an additive form (which we used in Step II); and to suggest using the discrete version of the majorant defined in \cite{SinnamonMonotonicity}, which allows the `regularization' process (Step III), and finally, to suggest treating separately the cases $i(m)=0,1,\geq 2$ (as carried out by the author in \eqref{eq:Jordan10},\eqref{eq:Jordan11},\eqref{eq:Jordan12}). 

The overall strategy of the proof and organization and flow of the arguments were developed by the author. The proofs, and the entire paper, were written by the author, who is responsible for their validity.

\section{Preliminaries}

Let $\Bbbk$ be a field of characteristic different from two. Let $X=\{x_1,\dots,x_d\}$ and let $X^*=\bigcup_{n=0}^{\infty} X^n$ (where $X^0=\{1\}$ is the empty word). Let $\mathcal{F} = \Bbbk\langle X \rangle$ be the free associative algebra generated by $X$ and let $\mathfrak{F} \subset \mathcal{F}$ be the free Lie algebra generated by $X$, viewed as a subspace of $\mathcal{F}$. We denote the natural $\mathbb{N}$-gradings of $\mathcal{F}$ and $\mathfrak{F}$ by $\mathcal{F}=\bigoplus_{n=0}^{\infty} \mathcal{F}_n$ and $\mathfrak{F}=\bigoplus_{n=0}^{\infty} \mathfrak{F}_n$, respectively. 
Given a subspace $V$ of an associative algebra, let $V^n$ denote the subspace spanned by all $n$-fold products of elements from $V$ and $V^{\leq n} = V+V^2+\dots+V^n$. For a Lie algebra, we use the notation $V^{[n]}$ and $V^{[\leq n]}$ to stress that we are considering Lie brackets.

Given a subset $S\subseteq X^*$, denote by $S^{\her}$ its hereditary closure, namely, the set of all factors of words from $S$. Given a hereditary language $\mathcal{X}\subseteq X^*$, $\Span_\Bbbk \left( X^* \setminus \mathcal{X} \right) \triangleleft \Bbbk\langle X\rangle$, and we denote by $A_\mathcal{X} = \Bbbk\langle X\rangle / \Span_\Bbbk \left( X^* \setminus \mathcal{X} \right)$ the monomial algebra associated with $\mathcal{X}$; this algebra is graded by the word length, and we denote $A_\mathcal{X} = \bigoplus_{n=0}^{\infty} \left( A_\mathcal{X} \right)_n$. The growth function of $A_\mathcal{X}$ (with respect to $X$) is given by counting the number of words in $\mathcal{X}$ of length at most $n$.

Given an associative algebra $A$, we denote by $A^{(-)}=(A,[\cdot,\cdot])$ the Lie algebra structure on $A$ given by the Lie brackets $[x,y]=xy-yx$, and $[A,A]\subseteq A^{(-)}$ is the Lie subalgebra spanned by all Lie brackets of elements from $A^{(-)}$. If $A$ is generated by a finite number of nilpotent elements then $[A,A]$ is finitely generated and shares the same growth function as $A$ \cite{AlahmadiAlharthi}. 

Recall that a Jordan algebra is a non-associative algebra satisfying $x\circ y=y\circ x$ (commutativity) and $(x\circ x)\circ (x\circ y)=x\circ((x\circ x)\circ y)$ (Jordan identity). 
Given an associative algebra $A$, we denote by $A^{(+)}=(A,\circ)$ the Jordan algebra structure on $A$ given by $x\circ y = \frac{1}{2}(xy+yx)$.

If $A$ is a finitely generated associative, Lie, or Jordan algebra, generated by a finite set $X$, then it admits a natural filtration by $F_n(A)=\Span_\Bbbk X^{\leq n}$. The associated graded algebra $\gr_X(A) = \bigoplus_{n=0}^{\infty} F_{n}(A)/F_{n-1}(A)$ (where $F_{-1}(A)=\{0\}$ and $F_0(A)=\{0\}$ in the case of Lie algebras, and $F_0(A)=\Bbbk$ in the case of unital associative/Jordan algebras) is $\mathbb{N}$-graded and generated by its degree-$1$ component.

\section{Lie algebras}

\subsection{Lie algebras of polynomial growth} In this part, we prove Theorem \ref{thm:Lie}. We start with a brief description of the construction idea.

\medskip

\noindent \emph{Proof Sketch.}
The core idea is to `pump' a Lie algebra $H$ arising from an associative monomial algebra acting on a monomial module $M$. Monomial modules of prescribed growth are relatively easy to construct, and can be glued together with $H$ by a semidirect product type construction. The main point is to ensure that the dominant contribution to the growth of the resulting Lie algebra comes from the module part, rather than from the acting Lie algebra $H$. 
This is possible thanks to an inherent deficit in the growth of the relevant Lie subalgebras compared with that of their ambient associative algebras. In short, and in a very simplified form, if we consider a monomial associative algebra generated by $x,y$, in which the non-zero monomials are precisely those containing at most $q$ occurrences of $y$, its growth is $\sim n^{q+1}$. But the growth of the Lie subalgebra generated by $x,y$ is strictly smaller: when counting left-normed iterated Lie brackets of length at least two, we may assume, up to sign, that they start with $y$ (otherwise, they start with $[x,x]=0$). 
Loosely speaking, this allows us to `work in strips': we construct the module $M$ and the associated Lie algebra $H$ inductively, each time pumping the module to match the magnitude of the desired growth function, assuming it behaves locally between $n^q$ and $n^{q+1}$. We take advantage of the fact that we can realize $H$ locally as a subalgebra of a monomial algebra as above, with at most $q$ occurrences of $y$, thus restricting its growth to $O(n^q)$, which is dominated by the module part.

\medskip

Following \cite{GZ_module}, we define monomial modules. 
Let $S(0)=\{1\}$ and, for each $n\geq 1$, let $S(n)\subseteq X\cdot S(n-1)$. Thus $\mathcal{S} = \bigcup_{n\geq 0} S(n) \subseteq X^*$ is a (not necessarily hereditary) formal language closed under taking suffixes, and it forms a $\Bbbk$-linear basis of a graded cyclic left $\mathcal{F}$-module $M=\Span_\Bbbk \mathcal{S}$. 
The growth function of $M$ is given by $\gamma_M(n)=\sum_{i=0}^{n} |S(i)|$. 

Let $\psi \colon \mathcal{F} \rightarrow \End_\Bbbk(M)$ be the natural homomorphism induced by the action $\mathcal{F} \curvearrowright M$. 
Let $H = \psi\left(\mathfrak{F}\right)$ and denote by $S_{x_1},\dots,S_{x_d}$ the images of $x_1,\dots,x_d$ under $\psi$. Notice that $H$ is generated as a Lie algebra by $S_{x_1},\dots,S_{x_d}$. Let $U=\Span_\Bbbk \{S_{x_1},\dots,S_{x_d}\}$ and let 
\[ 
\gamma_H(n) = \dim_\Bbbk U^{[\leq n]} = \dim_\Bbbk \psi\left(\bigoplus_{i=1}^{n} \mathfrak{F}_i \right)
\]
be the growth function of $H$. 

We define a Lie algebra $L$ which, as a vector space, is
\[
L = H \oplus M
\]
with Lie brackets given by $[(h,m),(h',m')]=([h,h'],h(m')-h'(m))$. 
In other words, $L$ is the semidirect product $H\ltimes M$. Notice that we can view 
\[
L \cong \left(\begin{matrix} H & M \\ 0 & 0 \end{matrix} \right) \subseteq \left(\begin{matrix} \End_\Bbbk(M) & M \\ 0 & 0 \end{matrix} \right)
\]
with Lie brackets induced by the associative structure of the latter.
\begin{prop} \label{prop:growth L}
The Lie algebra $L$ is generated by 
\[ 
\sigma_i := (S_{x_i},0),\ i=1,\dots,d;\ e := (0,1)
\]
and the growth function of $L$ with respect to this generating set satisfies
\[ \gamma_{L}(n) = 
\gamma_H(n) + \gamma_M(n-1). \]
\end{prop}

\begin{proof}
Let $V=\Span_\Bbbk \{\sigma_{1},\dots,\sigma_{d},e\}$. We prove that
\begin{equation} \label{eqqq}
V^{[\leq n]} =  U^{[\leq n]} \oplus  \Span_\Bbbk \bigcup_{i=0}^{n-1} S(i).
\end{equation}
For $n=1$ this is clear by definition of $V$. Given $u\in U^{[\leq n]},w\in S(i),0\leq i\leq n-1,1\leq j\leq d$ we have
\[
[ (0,1) , (u , w)] =  (0 , -u\cdot 1) \in 0\oplus \Span_\Bbbk \bigcup_{i=0}^{n} S(i) 
\]
\[
[ (S_{x_j} , 0) , ( u , w ) ] =  ( [S_{x_j},u] , x_j \cdot w) \in  U^{[\leq n+1]} \oplus  \Span_\Bbbk \bigcup_{i=0}^{n} S(i)
\]
and further, since 
\[
U^{[\leq n+1]}=U^{[\leq n]}+[S_{x_1},U^{[\leq n]}]+\dots+[S_{x_d},U^{[\leq n]}],
\]
and $S(n)\subseteq x_1 S(n-1)\cup \dots \cup x_d S(n-1)$, \eqref{eqqq} follows. Consequently, $\gamma_{L}(n) = \gamma_H(n) + \gamma_M(n-1)$.
\end{proof}

Let $\varphi\colon \mathcal{F} \rightarrow A_{\mathcal{S}^{\her}}$ denote the natural surjection.

\begin{prop} \label{prop:growth H}
We have $\gamma_H(n) = \sum_{m=1}^{n} \dim_\Bbbk \varphi(\mathfrak{F}_m)$.
\end{prop}

\begin{proof}
Recall that $\gamma_H(n) = \dim_\Bbbk U^{[\leq n]},\ U=\Span_\Bbbk \{S_{x_1},\dots,S_{x_d}\}$. Consider the map $f\colon \mathcal{F} \rightarrow \End_\Bbbk(M)$ given by $f(x_1)=S_{x_1},\dots,f(x_d)=S_{x_d}$. We claim that $\Ker(f)=\Ker(\varphi)$ and hence $f=g\circ \varphi$ with an isomorphism $g$ onto the image of $f$, establishing the desired equality.

First, if $u=x_{i_1}\dots x_{i_k}\in \mathcal{F}$ is a monomial vanishing in $A_{\mathcal{S}^{\her}}$ (namely, $\varphi(u)=0$) then $u$ does not occur as a subword of any word from $\mathcal{S}$, and therefore $S_{x_{i_1}}\circ \dots \circ S_{x_{i_k}}(w)=x_{i_1}\dots x_{i_k} w$ cannot be in $\mathcal{S}$, so it is equal to zero in $M$; hence $f(u) =S_{x_{i_1}}\circ \dots \circ S_{x_{i_k}}=0$.

Conversely, suppose that $\xi = \sum_{i=1}^{t} c_i u_i$ lies in $\Ker(f)$, and is written as a sum of distinct monomials with non-zero coefficients. In particular, for every word $w\in X^*$, we have that $\sum_{i=1}^{t} c_i u_i w$ is zero in $M$. Since $\mathcal{S}$ forms a basis of $M$, and all $u_1w,\dots,u_tw$ are distinct, it follows that none of $u_1w,\dots,u_tw$ belongs to $\mathcal{S}$. 
Fix $1\leq i\leq t$ and suppose that $u_i$ occurs as a subword of some word $v\in \mathcal{S}$, say, $v=u' u_i u''$. Then $u_i u''\in \mathcal{S}$ as $\mathcal{S}$ is closed under suffixes, contradicting the previous assertion. Hence $u_1,\dots,u_t\in X^*\setminus \mathcal{S}^{\her}$, so they are all congruent to zero in $A_{\mathcal{S}^{\her}}$ and $\varphi(\xi)=0$.
\end{proof}

We now specialize to $d=2$ and work over the alphabet $X=\{x,y\}$.

Let
\[
\L(N,n) := \Big\{\substack{\text{Words of length}\ N\ \text{with any two occurrences} \\ \text{of}\ y\ \text{separated by at least}\ n-1\ \text{occurrences of}\ x}\Big\}
\]
and let $L(N,n)=|\L(N,n)|$.
Observe that, fixing $n$, $L(-,n)$ is monotone non-decreasing. Notice that $\L(N,n)$ is the collection of sparse words denoted $T(n-1,N)$ in \cite{BZ}.

\begin{lem} \label{lem:L}
For each $c\in \mathbb{N}$, we have 
\begin{enumerate}
    \item $L(cn,n)\leq (n+1)^c$
    \item $L(cn,n) \geq {n+c-1 \choose c} \geq \frac{n^c}{c!}$
    \item $\sum_{j=1}^{cn} L(j,n) \geq \frac{1}{(c+1)!} n^{c+1}$.
\end{enumerate}
\end{lem}

\begin{proof}
\begin{enumerate}
    \item Let $w\in \L(cn,n)$ and write $w=w_1\dots w_c$ with each $w_i$ of length $n$. Notice that each $w_i$ can contain at most one occurrence of $y$. We thus have at most $n+1$ options for the position of $y$ (including its absence): either as one of its $n$ letters, or none at all. In total, we have at most $(n+1)^c$ options.
    
    \item Given $1\leq a_1<\dots<a_c\leq n+c-1$, consider $a_1,a_2+(n-1),a_3+(2n-2),\dots,a_c+(c-1)(n-1)$ and notice that $a_c+(c-1)(n-1)\leq n+c-1+(c-1)(n-1) = cn$, and so the length-$cn$ word in $x,y$ with $y$ occurring in the above positions is in $\L(cn,n)$.

    \item Given $1\leq a_1<\dots<a_{c+1}\leq n+c$, consider a word of length $n+c$ with $y$ precisely in these positions. Space out $a_1,\dots,a_c$ by adding $n-1$ occurrences of $x$ between each two. Then truncate the resulting word by deleting the letters starting from the last occurrence of $y$. More formally, consider $b_i:=a_i+(i-1)(n-1)$ for all $1\leq i\leq c$, and $b_{c+1}:=a_{c+1}+(c-1)(n-1)$. Notice that $b_{c+1}\leq n+c+(c-1)(n-1)\leq cn+1$, so if we look at the word of length $b_{c+1}-1$ with $y$ in positions $b_1,\dots,b_c$, we obtain a word in $\bigcup_{j=1}^{cn} \L(j,n)$. This assignment is injective, proving the required inequality.
\end{enumerate}    
\end{proof}

Let $f\colon \mathbb{N}\rightarrow \mathbb{N}$ be an increasing, polynomially bounded function. Up to $\sim$, we may add to $f$ a linear term, making $f'(n)=f(n)-f(n-1)\geq 2$ for all $n$.

Fix $q\in \{2,3,4,\dots\}$ large enough such that, for some $N\in \mathbb{N}$, 
\[
f(n) <  \frac{1}{q^q q!} n^q
\]
for all $n\geq N$. 
Let $N_1 > \max\{N,q\}$ 
and notice that for every $n > q$ we have
\[
n < \frac{1}{2} n^2 <\frac{1}{6} n^3 <\dots < \frac{1}{q!} n^q.
\]
Now, for each $i$, we define $N_{i+1} = q N_i = q^{i} N_1$. 
Given $N_i$, let $\Delta_{i} := f(N_{i+1}) - f(N_i)$. Notice that 
\[ (q-1) N_i  \leq \sum_{j=N_i+1}^{N_{i+1}} f'(j) = \Delta_i \leq f(N_{i+1})=f(qN_i) < \frac{N_i^q}{q!}
\]
so there exists some $2\leq q_i\leq q$ such that 
\begin{equation} \label{eq:for intro}
\frac{1}{(q_i-1)!} N_i^{q_i-1} \leq \Delta_i < \frac{1}{q_i !}N_i^{q_i}.
\end{equation}
We find some $1 \leq X \leq (q_i - 1) N_i$ such that
\begin{equation} \label{eq:for intro2}
\sum_{j=1}^{X-1} L(j,N_i) \leq \Delta_i < \sum_{j=1}^{X} L(j,N_i).
\end{equation}
This is possible thanks to Lemma \ref{lem:L}, as 
\[
\sum_{j=1}^{(q_i - 1)N_i} L(j,N_i)\geq \frac{1}{q_i !} N_i^{q_i} > \Delta_i.
\]
And, on the other hand, $L(1,N_i)=| \{x,y\} |=2\leq f'(N_{i+1}) \leq \Delta_i$.

We pick an arbitrary subset $\L_0(X,N_i) \subseteq \L(X,N_i)$ such that
\[
\L(1,N_i) \cup \dots \cup \L(X-1,N_i) \cup \L_0(X,N_i)
\]
is of size $\Delta_i$. We let $\L_{N_1} = \{x^{N_1}\}$ and 
\begin{eqnarray*}
&&\L_{N_i + j} =  \L(j,N_i) x^{N_i}\subseteq \{x,y\}^{N_i+j}\ \text{for}\ 1\leq j\leq X-1 \\ 
&&\L_{N_i + X} =  \L_0(X,N_i) x^{N_i} \subseteq \{x,y\}^{N_i+X} \\
&&\L_{m} = \emptyset\ \text{for all}\ N_i+X+1\leq m\leq N_{i+1}.    
\end{eqnarray*}
(Notice that $X\leq (q_i-1)N_i$ so $N_i+X\leq q_i N_i\leq q N_i$.)

Let \[ \mathcal{S} := \{1,x,x^2,\dots\} \cup  \bigcup_{m=N_1}^{\infty} \L_{m} \]
\begin{claim}
The set $\mathcal{S}$ is right hereditary, that is, closed under suffixes.
\end{claim}
\begin{proof}
We prove that, for each $m>N_1$, the length $m-1$ suffix of every word in $\L_m$ belongs to $\L_{m-1} \cup \{x^{m-1}\}$. If $m=N_i+1$ then this follows since $\L_{N_i+1} \subseteq \{x,y\}\cdot x^{N_i}$. If $m=N_i+j,\ 1<j\leq X$ then every word in $\L_{N_i+j}$ takes the form $ux^{N_i}$ where $u$ is a length-$j$ word in which occurrences of $y$ are separated by at least $N_i-1$ occurrences of $x$. A suffix of $u$ also enjoys the same separation property, so the length-$N_i+j-1$ suffix of $ux^{N_i}$ lies in $\L(j-1,N_i)x^{N_i} = \L_{N_i+j-1}$. For $N_i+X < m \leq N_{i+1}$, we have $\L_m=\emptyset$ and the claim holds evidently.
\end{proof}

\begin{claim} \label{claim:growth M}
For every $i\geq 1$,
\[ f(N_i) - C \leq \gamma_M(N_i) \leq  f(N_i) + N_i + 1 \]
for some constant $C>0$.
\end{claim}
\begin{proof}
Recall that we picked $\L_0(X,N_i)$ such that 
\[
|\L(1,N_i)|+\dots+|\L(X-1,N_i)|+|\L_0(X,N_i)|=\Delta_i.
\]
Since $|\L_{N_i+j}|=|\L(j,N_i)|$ for all $1\leq j\leq X-1$ and $|\L_{N_i+X}|=|\L_0(X,N_i)|$, it follows that $\left|\bigcup_{m=N_i+1}^{N_{i+1}} \L_m\right|=\Delta_i$.
Hence
\begin{eqnarray*}
|\mathcal{S} \cap \bigcup_{j=0}^{N_i} \{x,y\}^j| & = & |\{1,x,\dots,x^{N_i}\} \cup \bigcup_{m=N_1+1}^{N_{i}} \L_m |\\ & = & \left( \sum_{j=1}^{i-1} \Delta_{j} \right) + r \\ & = & f(N_{i})-f(N_1) + r,
\end{eqnarray*}
for some $0\leq r\leq N_i+1$
and the claim follows for $C=f(N_1)$.
\end{proof}

\begin{claim} \label{claim:growth varphi}
We have 
\[ \sum_{d=N_i+1}^{N_{i+1}} \dim_\Bbbk \varphi(\mathfrak{F}_d) \leq c N_i^{q_i-1} \] and 
\[ \sum_{d=1}^{N_{i+1}} \dim_\Bbbk \varphi(\mathfrak{F}_d) \leq c' f(N_{i+1}) \] for some constants $c,c'>0$.
\end{claim}

\begin{proof}
Let $\mathcal{X} = \mathcal{S}^{\her}$ and consider the monomial algebra $A_\mathcal{X}$. Fix $i\geq 1$ and let $N_i+1\leq d\leq N_{i+1}$. Every non-zero length-$d$ monomial $u \in A_\mathcal{X}$ is a subword of a word $w\in \L(j,N_k)x^{N_k}$ for some $k\geq i$ and $1\leq j\leq N_{k+1}-N_k$. Furthermore, if $k\geq i+1$ then every two occurrences of $y$ in $w$ are separated by at least $N_{i+1}-1$ occurrences of $x$, so $u$ cannot contain more than one occurrence of $y$. For $k=i$, write $w=vx^{N_i}$ for $v\in \L(j,N_i)$, $j\leq X\leq (q_i-1)N_i$, so $w$ can contain at most $q_i-1$ occurrences of $y$. It follows that every non-zero monomial in $A_\mathcal{X}$ of length $\in [N_i+1,N_{i+1}]$ has at most $q_i - 1$ occurrences of $y$.

Let $\xi = [a_1,\dots,a_d]=[\cdots[[a_1,a_2],a_3],\dots]$ be an iterated Lie bracket with $a_1,\dots,a_d\in \{x,y\}$. Suppose that exactly $s$ of these are equal to $y$. Notice that, writing $\xi = \sum \alpha_j w_j$ as a linear combination of monomials in $\mathcal{F}$ of length $d$, each $w_j$ contains exactly $s$ occurrences of $y$. Recall that $\mathfrak{F}_d$ is spanned by all $d$-fold iterated Lie brackets and by the above argument, $\varphi(\xi)=0$ for every $d$-fold Lie bracket $\xi=[a_1,\dots,a_d]$ with at least $q_i$ occurrences of $y$. We may assume that $a_1\neq a_2$ and further, up to sign, $a_1=y$. Hence the number of distinct such brackets with at most $q_i-1$ occurrences of $y$ is at most
\[
\sum_{p=1}^{q_i - 1} {d-2 \choose p-1} \leq \sum_{p=1}^{q_i - 1} (d-2)^{p-1} \leq q_i d^{q_i-2}
\]
Therefore
\begin{eqnarray*}
    \sum_{d=N_i+1}^{N_{i+1}} \dim_\Bbbk \varphi(\mathfrak{F}_d) & \leq & \sum_{d=N_i+1}^{N_{i+1}} q_i d^{q_i - 2} \\ & \leq & (N_{i+1} - N_i) \cdot q N_{i+1}^{q_i - 2} \\ & \leq & q^2 N_i \cdot q \cdot q^{q-2} N_i^{q_i-2} \\ & \leq & q^{q+1} N_i^{q_i-1}
\end{eqnarray*}
and the first claim follows for $c=q^{q+1}$. Next, recall that for all $1\leq j\leq i$ we have $\frac{1}{(q_j-1)!}N_j^{q_j-1} \leq \Delta_j$, so $N_j^{q_j-1}\leq (q-1)! \Delta_j$ and 
\begin{eqnarray*}
\sum_{d=N_1+1}^{N_{i+1}} \dim_\Bbbk \varphi(\mathfrak{F}_d) & = & \sum_{j=1}^{i} \sum_{d=N_j+1}^{N_{j+1}} \dim_\Bbbk \varphi(\mathfrak{F}_d) \\ & \leq & c(q-1)!\sum_{j=1}^{i} \Delta_j \\ & = & c(q-1)!\left(f(N_{i+1}) - f(N_1)\right)
\end{eqnarray*}
and $\sum_{d=1}^{N_{i+1}} \dim_\Bbbk \varphi(\mathfrak{F}_d) \leq c' f(N_{i+1}) + c''$ for some constant $c'' = \sum_{d=1}^{N_1} \dim_\Bbbk \varphi(\mathfrak{F}_d)$, which can be absorbed into $c'$ by enlarging it if necessary. Hence the second inequality in the claim's assertion follows as well.
\end{proof}

We recall the following (standard) fact.

\begin{fact} \label{fact}
If $f,g\colon \mathbb{N}\rightarrow \mathbb{N}$ are non-decreasing and $f(n_i)\leq g(n_i)$ for an increasing sequence $n_1<n_2<\dots$ with bounded quotients $n_{i+1}/n_i \leq K$, then $f(n)\preceq g(n)$. Indeed, given $n>n_1$, let $i$ be such that $n_i < n \leq n_{i+1} \leq Kn_i$. Then $f(n) \leq f(n_{i+1}) \leq g(n_{i+1})\leq g(Kn_i)\leq g(Kn)$.
\end{fact}

\begin{proof}[{Proof of Theorem \ref{thm:Lie}}]
Consider the Lie algebra $L$ constructed above. By Proposition \ref{prop:growth L} and Proposition \ref{prop:growth H}, 
\[
\gamma_L(n) = \gamma_H(n)+\gamma_M(n-1) = \sum_{d=1}^{n} \dim_\Bbbk \varphi(\mathfrak{F}_d) + \gamma_M(n-1).
\]
By Claim \ref{claim:growth M},
\[
\gamma_M(N_i-1)\leq \gamma_M(N_i) = f(N_i) + O(N_i).
\]
Notice that $\gamma_M(n+1)\leq 3 \gamma_M(n)$, since $\bigcup_{i=0}^{n+1} S(i) \subseteq \bigcup_{i=0}^{n} S(i) \cup x \bigcup_{i=0}^{n} S(i) \cup y \bigcup_{i=0}^{n} S(i)$, so
\[
\gamma_M(N_i-1)\geq \frac{1}{3} \gamma_M(N_i) \geq \frac{1}{3} f(N_i) - C.
\]
Therefore,
\[
\gamma_L(N_i) \leq \sum_{d=1}^{N_i} \dim_\Bbbk \varphi(\mathfrak{F}_d) + f(N_i) + O(N_i) \overset{\text{Claim\ }\ref{claim:growth varphi}}{\leq}  A\cdot f(N_i)  
\]
for some constant $A>0$, and
\[
\gamma_L(N_i) \geq \gamma_M(N_i-1) \geq \frac{1}{3} f(N_i) - C \geq \frac{1}{4} f(N_i)
\] 
where the last inequality holds for all $i\gg 1$, as $f(n)\xrightarrow{n\rightarrow \infty} \infty$. 
Since $N_{i+1}/N_i=q$ is bounded, it follows by Fact \ref{fact} that $\gamma_L(n) \sim f(n)$.
\end{proof}

\subsection{Associative algebras with polynomial identities}
We give an associative counterpart of Theorem \ref{thm:Lie}. The growth functions of associative algebras have been completely characterized up to asymptotic equivalence by Bell and Zelmanov \cite{BZ}. The construction below is similar in the sense that it is built upon constructing a monomial algebra consisting of monomials in two letters with sparse occurrences of one of the letters; the approach below shows that, in the polynomial realm, one need not a priori assume the discrete derivative condition in the Bell--Zelmanov theorem. 

It is probably possible to reach the same conclusion by first showing, by means of discrete calculus, that increasing polynomially bounded functions that grow at least quadratically are bi-Lipschitz equivalent to ones satisfying the conditions of the Bell--Zelmanov theorem, and then utilizing their construction, keeping the number of occurrences of one of the letters bounded; we give the construction below as a more concrete realization, and since it is relatively quick given the ideas developed before. This further underscores the difference in the difficulty between describing the space of growth functions of Lie algebras compared to associative algebras, where one has access to a rich toolkit from combinatorics on words. Notice that the result below also strengthens \cite[Theorem~1.2]{GZ_module}.

\begin{thm} \label{thm:associative}
    Let $f\colon \mathbb{N}\rightarrow \mathbb{N}$ be an increasing, polynomially bounded function such that $f(n)\succeq n^2$. Then there exists a finitely generated associative $\Bbbk$-algebra $A$, satisfying a polynomial identity, such that $\gamma_A(n) \sim f(n)$.
\end{thm}

\begin{proof}
Let $f\colon \mathbb{N}\rightarrow \mathbb{N}$ be an increasing, polynomially bounded function, such that $f(n)\succeq n^2$. We may assume that $f(n) \geq {n+2 \choose 2}$. 
Let $q \in \{2,3,4,\dots\}$ be large enough such that, for some $N\in \mathbb{N}$, \[
f(n) \leq \frac{n^q}{q^q q!}
\]
for all $n\geq N$. We work over the alphabet $\{x,y\}$, and retain the notation $L(N,n)$ from the previous subsection. 

As before, we define $N_1=N$ and $N_{i+1}=qN_i=q^iN_1$ and $\Delta_i := f(N_{i+1})-f(N_i)$.
For each $m$, we define a subset $\L_m\subseteq \{x,y\}^{m}$ such that the following \emph{hereditary property} holds: every length-$(m-1)$ subword of $\L_{m}$ belongs to 
\[ \L_{m-1} \cup \{x^j y x^{m-j-2}\ :\ 0\leq j\leq m - 2\} \cup \{x^{m-1}\}. \]

We take $\L_m = \emptyset$ for all $m\leq N_1$.
Suppose $\L_1,\dots,\L_{N_i}$ have been defined.

Notice that 
\[
\Delta_i < f(qN_i) \leq \frac{N_i^q}{q!} \overset{\ref{lem:L}}{\leq} L(qN_i,N_i)
\]
Therefore there exists some $X$ in $[N_i+1, qN_i]$ for which 
\begin{equation} \label{eqq:2}
\sum_{j=N_i+1}^{X-1} L(j,N_i) \leq \Delta_i < \sum_{j=N_i+1}^{X} L(j,N_i).
\end{equation}

For each $N_i+1\leq j < X$, we take $\L_{j}=\L(j,N_i)$. Thus $|\L_j|=L(j,N_i)$. We take $\L_X$ to be an arbitrary subset of $\L(X,N_i)$ such that
\[
\L(N_i+1,N_i)\cup \dots\cup \L(X-1,N_i) \cup \L_X
\]
is of size $\Delta_i$, which is possible by \eqref{eqq:2}. For $X<j\leq N_{i+1}=qN_i$, take $\L_j=\emptyset$.

We now verify the hereditary property for $\{\L_m\}_{m=1}^{\infty}$. If $m\in [N_i+1,N_{i+1}]$ and $m>X$ then $\L_m=\emptyset$, and there is nothing to prove; if $N_{i}+2\leq m \leq X$ and $w\in \L_m$ then every length $m-1$ subword $w'$ of $w$ still has every two occurrences of $y$ separated by at least $N_i-1$ occurrences of $x$, so $w'\in \L(m-1,N_i)=\L_{m-1}$. If $m=N_i+1$ then every $w\in \L_m \subseteq \L(m,N_i)$ either contains at most one occurrence of $y$ or takes the form $yx^{N_i - 1}y$, and hence every length $m-1$ subword $w'$ of $w$ has at most one occurrence of $y$ and thus belongs to $\{x^j y x^{m-j-2}\ :\ 0\leq j\leq m-2\}\cup\{x^{m-1}\}$, and the required property holds. 
Hence \[ \mathcal{X} := \{1,x,x^2,\dots\} \cup \{x^j y x^{k}\ :\ j,k \geq 0\} \cup \bigcup_{m=1}^{\infty} \L_{m} \]
is a hereditary formal language, giving rise to a monomial algebra $A_\mathcal{X}$ whose growth function is given by counting the words in $\mathcal{X}$ of length at most $n$.

We have
\begin{eqnarray} \label{eqq:5}
|\L_1|+\dots+|\L_{N_i}| & = & \sum_{j=1}^{i-1} \left( \sum_{t=N_j+1}^{N_{j+1}} |\L_t| \right) = \sum_{j=1}^{i-1} \Delta_j  = f(N_i)-f(N_1) 
\end{eqnarray}
and
\begin{equation} \label{eqq:5.5}
|\{1,x,\dots,x^{n}\} \cup \{x^j y x^k\ :\ j+k\leq n-1\}|=(n+1)+\sum_{j=0}^{n-1} (n-j) = {n+2 \choose 2}.
\end{equation}
Hence
\[
 f(N_i)-f(N_1) \overset{\eqref{eqq:5}}{\leq} \gamma_{A_\mathcal{X}}(N_i) \overset{\eqref{eqq:5},\eqref{eqq:5.5}}{\leq} {N_i+2 \choose 2} + f(N_i) - f(N_1) \leq 2f(N_i).
\]
Since $f$ is unbounded, for $i\gg 1$, it holds that $f(N_i)-f(N_1)\geq \frac{1}{2}f(N_i)$. Hence $\frac{1}{2}f(N_i) \leq \gamma_{A_\mathcal{X}}(N_i) \leq 2f(N_i)$ and by Fact \ref{fact}, $\gamma_{A_\mathcal{X}}(n) \sim f(n)$.

Finally, observe that every word in $\mathcal{X}$ contains at most $q$ occurrences of $y$ and hence $A_\mathcal{X}$ is a homomorphic image of $\Bbbk \langle x,y \rangle  / \langle y \rangle ^{q+1}$, and therefore satisfies the polynomial identity $[X_1,Y_1]\cdots[X_{q+1},Y_{q+1}]=0$.
\end{proof}

\subsection{Super-polynomial gaps}
The smooth characterization of the space of growth functions as done in Theorem \ref{thm:Lie} in the polynomial realm, does not extend beyond polynomial growth.

\begin{prop} \label{prop:Liegap}
    There exists an increasing, subexponential function $f\colon \mathbb{N}\rightarrow \mathbb{N}$ that is not bi-Lipschitz equivalent to the growth function of any finitely generated Lie algebra.
\end{prop}

\begin{proof}
We define a function as follows. We let $c=1,2,3,\dots$ and define a sequence $N_c$, along with the values of $f$, inductively. Let $N_1=1$ and $f(1)=1$. We will have $f'(i)=1$ for all $i\notin \{ c^2 N_c\ :\ c=2,3,\dots\}$ and $f'(c^2 N_c) = \lceil \exp(N_c^{3/4}) \rceil$ for all $c=2,3,\dots$. Notice that $f$ is increasing by design, and has subexponential growth; in fact, $f(n) = O(\exp(n^{3/4}))$. 
To make the construction well-defined, we need to specify $N_{c}$ given that $N_{c-1}$ has been fixed. 

Toward this end, we recall the connection between the growth functions of Lie algebras and the growth functions of their universal enveloping algebras, discovered by M.~Smith \cite{Smith76} and significantly further developed by Petrogradsky \cite{Petrogradsky}. Given a function $\varphi\colon \mathbb{N}\rightarrow \mathbb{N}$, let $P_\varphi$ be given by
\[
\prod_{i=1}^{\infty} \frac{1}{(1-t^i)^{\varphi(i)}} = \sum_{n=0}^{\infty} P_\varphi(n) t^n.
\]
Then $P_{\gamma_L'}=\gamma_{U(L)}'$ (see \cite{Smith76}). Furthermore, Petrogradsky proved that if $\limsup_{n\rightarrow \infty} \log_n \varphi(n) \leq \alpha-1$ then $\limsup_{n\rightarrow \infty} \log_n \ln P_\varphi(n) \leq \frac{\alpha}{\alpha+1}$.
This means, in particular, that for every $\lambda>0$ there exists some $M=M_\lambda \in \mathbb{N}$ such that, if $\varphi(n)\leq \lambda n$ for every $n$ then $P_\varphi (n)\leq \exp(n^{0.7})$ for every $n\geq M$. Since $P_\varphi(N)$ depends only on the values of $\varphi$ up to $N$, we may assume that $\varphi(n)\leq \lambda n$ for all $n\leq N$ and conclude $P_\varphi (N)\leq \exp(N^{0.7})$ --- as long as $N\geq M$.

We can now specify $N_{c}$ given that $N_{c-1}$ has been determined. Assume that $N_1,\dots,N_{c-1}$ and the values of $f$ on $[1,(c-1)^2 N_{c-1}]$ are all fixed. Recall that for $i$ not of the form $j^2 N_j$, $f'(i)=1$. It follows that, letting $N_{c}\rightarrow \infty$, the function $f$ grows linearly, and so does $g(x):=cf(cx)$, for all $x < c N_c$. Hence, by the above analysis, if we take $N_c \gg N_{c-1}$, we can ensure that
\begin{equation} \label{eq:Liegap1}
P_g(N_c)\leq \exp(N_c^{0.7})    
\end{equation}
Further assume that $c < \log N_c$ (for $c>1$).

\medskip

Assume, toward contradiction, that there exists a finitely generated Lie algebra $L$ (which we may assume to be $\mathbb{N}$-graded, generated in degree $1$) such that $\gamma_L\sim f$, and let $K>1$ be such that $\gamma_L(n)\leq Kf(Kn),f(n)\leq K\gamma_L(Kn)$ for all $n$.
We have, for every $c \geq K$,
\begin{eqnarray} \label{eq:Liegap1.5}
    \frac{1}{c} f(c^2 n) \leq \frac{1}{K} f(c^2 n) & \leq & \gamma_L(K c^2 n) \\ & \leq & \gamma_L(c^3 n)\leq \gamma_{U(L)}(c^3 n)\leq \gamma_{U(L)}(n)^{c^3} \nonumber
\end{eqnarray}
where the last inequality utilizes the submultiplicativity of $\gamma_{U(L)}$, $U(L)$ being an associative algebra.

However, $\gamma_L(i)\leq Kf(Ki)\leq c f(ci)=g(i)$ for all $i\in [1,N_c]$. Hence \begin{equation} \label{eq:Liegap2}
P_{\gamma_L}(N_c) \leq P_g(N_c) \overset{\eqref{eq:Liegap1}}{\leq} \exp(N_c^{0.7}).
\end{equation}
Since $\gamma_{U(L)}'(N_c) = P_{\gamma_{L}'}(N_c) \leq P_{\gamma_L}(N_c)$, and since $\gamma_{U(L)}'$ is non-decreasing (e.g.~since it is a graded domain, and multiplication by any non-zero element of degree one induces an injective linear map $U(L)_d\rightarrow U(L)_{d+1}$), we have
\begin{eqnarray} \label{eq:Liegap3}
    \gamma_{U(L)}(N_c) & \leq & (N_c+1) \gamma_{U(L)}'(N_c) \\ & \leq & (N_c+1) P_{\gamma_L}(N_c) \nonumber \\ & \overset{\eqref{eq:Liegap2}}{\leq} & (N_c+1) \exp(N_c^{0.7}). \nonumber
\end{eqnarray}

Taking $n=N_c$ in \eqref{eq:Liegap1.5} and combining with the definition of $f'(c^2 N_c)$ and with \eqref{eq:Liegap3},
\begin{eqnarray*}
\frac{1}{c}\exp(N_c^{0.75}) \leq \frac{1}{c} f(c^2 N_c) & \overset{\eqref{eq:Liegap1.5}}{\leq} & \gamma_{U(L)}(N_c)^{c^3} \\ & \overset{\eqref{eq:Liegap3}}{\leq} & \left((N_c+1) \exp(N_c^{0.7})\right)^{c^3} \leq (N_c+1)^{\log^3 N_c} \exp(N_c^{0.7} \log^3 N_c)
\end{eqnarray*}
which leads to a contradiction when $c\rightarrow \infty$. This proves that $f$ is not bi-Lipschitz equivalent to the growth function of any Lie algebra.
\end{proof}

One might wonder if there exist additional natural features of growth functions of Lie algebras that can be employed to characterize them up to bi-Lipschitz equivalence. For instance, we always have $\gamma_L(n+1)\leq (c+1)\gamma_L(n)$ where $c$ is the size of the generating set. However, every increasing exponentially bounded function is bi-Lipschitz equivalent to a function with bounded successive quotients (e.g.~by \cite{GZ_module}), so this does not provide a genuine new (asymptotic) obstruction.

The obstruction utilized in Proposition \ref{prop:Liegap} uses the submultiplicativity of the universal enveloping algebra. More generally, viewing $L$ as a $U(L)$-module (via the adjoint representation), we obtain that $\gamma_L(n+m) \leq \gamma_{U(L)}(n) \cdot \gamma_L(m)$; hence, $\gamma_L$ is restricted by its past values, through the growth of its universal enveloping algebra. However, due to the generalized partition function interpretation of $\gamma_{U(L)}$, this feature fails to produce (increasing) functions not bi-Lipschitz equivalent to growth functions of Lie algebras beyond the stretched exponential border $\exp(\sqrt{n})$.

\section{Jordan algebras}

In this part, we prove Theorem \ref{thm:Jordan}. We start with a brief description of the proof idea.

\medskip

\noindent \emph{Proof Sketch.} The proof is quite implicit, in the sense that we do not associate a canonical/functorial associative algebra with any Jordan algebra, having the same growth rate. First, we use the action of the multiplicative envelope $M(J)$ on $J$. Using a multilinearization of the Jordan identities, we obtain (\`a la Jacobson) a submultiplicativity-like inequality $\gamma_J(n+m) \preceq \gamma_{M(J)}(n)^2 \gamma_J(m)$. Next, using a result of Mart\'inez--Zelmanov, $\gamma_{M(J)}(n)\preceq \gamma_J(n)^2$. The resulting submultiplicativity-type property of $\gamma_J(n)$ turns out to be sufficient to `regularize' it to a submultiplicative function, which is then bi-Lipschitz equivalent to a function satisfying the conditions of the Bell--Zelmanov theorem.

\medskip

\begin{proof}[{Proof of Theorem \ref{thm:Jordan}}]
\ 
\smallskip

\noindent \emph{Step I: $J-M(J)$ submultiplicativity.} This step uses standard techniques from the theory of Jordan algebras, and is likely well-known to experts. We give a complete argument for the reader's convenience.

Suppose $J$ is generated, as a Jordan algebra, by $X=\{x_1,\dots,x_d\}$. Given $a\in J$, we denote by $L_a\colon J\rightarrow J$ the left multiplication operator $L_a(x)=ax$.

Consider the free Jordan algebra $FJ(X)=\bigoplus_{n=0}^{\infty} FJ_n$ where $FJ_n$ is the $\Bbbk$-subspace spanned by all Jordan words in $X$ of length $n$. Let $FJ_{\leq n}=\bigoplus_{i=0}^{n} FJ_i$.
A known consequence of a polarization of the Jordan identity yields \cite{Jacobson,Martinez}
\begin{equation} \label{eq:Jordan1}
L_{(ab)c} = L_{ab} L_c + L_{bc} L_a + L_{ac} L_b - L_b L_c L_a - L_a L_c L_b.
\end{equation}
Let $S_1 = \{L_{x_i}\ :\ 1\leq i\leq d\}$ and $S_2 = \{L_{x_i x_j}\ :\ 1\leq i,j \leq d\}$ and set the \emph{weight} $\wt$ of operators from $S_1$ to be $1$, and the weight of operators from $S_2$ to be $2$. We claim that if $w$ is a Jordan monomial of degree $n$ in $X$ then
\begin{equation} \label{eq:Jordan2}
L_w \in \Span_\Bbbk \{ T_1\cdots T_s\ :\ T_1,\dots,T_s\in S_1\cup S_2,\ \sum_{j=1}^{s} \wt(T_j) = n \}. 
\end{equation}
For $n=1$ and $n=2$ this is obvious, as then $w=x_i$ or $w=x_i x_j$ (respectively), so $L_w\in S_1$ or $L_w\in S_2$ (resp.) and the claim holds. For $n\geq 3$, write $w=(ab)c$ where $a,b,c$ are non-trivial Jordan monomials of total degree $n$. By \eqref{eq:Jordan1}, $L_w=L_{(ab)c}$ is a linear combination of operators of the form $L_{pq}L_r$ and $L_p L_q L_r$ with $\{p,q,r\}=\{a,b,c\}$. Since $p,q,r,pq$ are all Jordan monomials of degree smaller than $n$, the induction hypothesis implies that $L_{pq}$ is a linear combination of products of operators from $S_1\cup S_2$ of total weight $\deg(p)+\deg(q)$, and $L_r$ is a combination of such products with total weight $\deg(r)$, and hence, $L_{pq}L_r$ is a combination of products of operators from $S_1\cup S_2$ of total weight $\deg(p)+\deg(q)+\deg(r)=\deg(a)+\deg(b)+\deg(c)=n$. The argument for $L_p L_q L_r$ is similar, writing each one of $L_p,L_q,L_r$ as a combination of products of total weight $\deg(p),\deg(q),\deg(r)$, respectively.

Now given a Jordan monomial $w\in FJ_n$, we can write $w=L_{u_1}L_{u_2}\dots L_{u_t}(x_i)$ where $u_1,\dots,u_t$ are Jordan monomials (not necessarily generators) of total degree $\deg(u_1)+\dots+\deg(u_t)=n-1$, and $x_i\in X$. (In terms of the binary tree representation of any monomial in the free non-associative algebra, $x_i$ corresponds to -- any -- leaf.) Now using \eqref{eq:Jordan2}, we can write each $L_{u_i}$ as a linear combination of products of operators from $S_1\cup S_2$ of total weight $\deg(u_i)$. Altogether,
\begin{equation} \label{eq:Jordan3}
w \in \Span_\Bbbk \{T_1\cdots T_s(x_i)\ :\ T_1,\dots,T_s\in S_1\cup S_2,\ \sum_{j=1}^{s} \wt(T_j)=n-1 \}.
\end{equation}
Let $V=\Span_\Bbbk\left(S_1 \cup S_2 \cup \{\idd\}\right) \leq M(FJ)$. Notice that $V\subseteq V^2\subseteq \dots$ and, by \eqref{eq:Jordan2}, $V$ generates $M(FJ)$. Now let $n,m\in \mathbb{N}$ be given and let $w\in FJ_{n+m}$ be a Jordan monomial. As in \eqref{eq:Jordan3}, we can write $w$ as a linear combination of expressions of the form $T_1\dots T_s(x_i)$ with $T_1,\dots,T_s\in S_1\cup S_2$ and $\sum_{j=1}^{s} \wt(T_j)=n+m-1$. Fix any of these expressions. Decompose $T_1\dots T_s=(T_1\dots T_{j-1})(T_j\dots T_s)$ for $1\leq j\leq s+1$ the smallest possible for which $\wt(T_j)+\dots+\wt(T_s)\leq m-1$ (if $j=s+1$, we interpret $`T_{s+1}\dots T_s'$ as the identity, of total weight zero). Thus $T_j\dots T_s(x_i)\in FJ_{\leq m}$. Now since $\wt(T_j)+\dots+\wt(T_s)\leq m-1$ and the weight of each $T_1,\dots,T_s$ is either $1$ or $2$, there are only two options: either $\wt(T_j)+\dots+\wt(T_s)=m-1$, in which case $\wt(T_1)+\dots+\wt(T_{j-1})=n$ and so $j-1\leq n$ and therefore $T_1 \cdots T_{j-1}\in V^n$; or $\wt(T_j)+\dots+\wt(T_s)=m-2$ and $\wt(T_{j-1})=2$, in which case $\wt(T_1)+\dots+\wt(T_{j-1})=n+1$, but again $j-1\leq n$ and $T_1\cdots T_{j-1}\in V^n$. In all cases, $T_1\cdots T_s(x_i)\in V^n \cdot FJ_{\leq m}$. Here $\cdot$ stands for the operation of a linear transformation from $M(FJ)$ on $FJ$. Since $w$ was a linear combination of such expressions, we conclude that $FJ_{\leq n+m} \subseteq V^n \cdot FJ_{\leq m}$. Finally, pushing this along a surjective homomorphism $FJ\twoheadrightarrow J$, we obtain the same inclusion for $J$. We let $J_{\leq n}$ denote the subspace of $J$ spanned by all products of at most $n$ generators, $V$ the subspace of $M(J)$ spanned by $L_{x_i},L_{x_i x_j},\idd$ and conclude
\begin{equation} \label{eq:Jordan4}
J_{\leq n+m} \subseteq V^n \cdot J_{\leq m}.
\end{equation}
(We formally proved the inclusion for Jordan monomials of length exactly $n+m$; for words of length $k+m,\ k<n$, the assertion follows similarly as $V^k\subseteq V^n$, and for Jordan monomials of length $\leq m$, the assertion is evident.)
Consequently, letting $\gamma_J(n) = \gamma_{J,X}(n) = \dim_\Bbbk J_{\leq n}$ and $\gamma_{M(J)}(n) = \gamma_{M(J),V}(n) = \dim_\Bbbk V^n$, we have
\[
\gamma_J(n+m) \leq \gamma_{M(J)}(n) \gamma_J(m).
\]
By \cite[Lemma~2]{MZ22}, we have $\gamma_{M(J)}(n)\preceq \gamma_J(n)^2$. In fact, the proof shows that $\gamma_{M(J)}(n) \leq C\gamma_J(2n)^2$ for some $C>1$ (which we may assume to be an integer), and hence
\[
\gamma_J(n+m) \leq C \gamma_{J}(2n)^2 \gamma_J(m).
\]
Replacing $\gamma_J(n)$ by $f(n) := C \gamma_J(n)$ -- which is bi-Lipschitz equivalent to it -- we get
\[
f(n+m)\leq f(2n)^2 f(m).
\]
We further adjust $f$ to `absorb' the inner doubling in the above inequality.
We put 
\begin{equation} \label{eq:Jordan4.5}
    \hat{f}(n):=\max\{f(n),\left\lceil \sqrt{\sup_{t\geq 1} \frac{f(n+t)}{f(t)}} \right\rceil\}.
\end{equation}
This definition is based on a discrete analog of a concept discussed in \cite{GrochenigWeights} (discussion on `moderate weights' before Lemma 5.1) in the continuous context, namely, $\sup_y \frac{f(x+y)}{f(y)}$. 

Note that $\hat{f}$ is well-defined, as $f(n+t)\leq f(2n)^2 f(t)$, and therefore $\hat{f}(n)\leq f(2n)$. Since $f(n)\leq \hat{f}(n)$ by design, $\hat{f}\sim f$. Notice that $\hat{f}$ is non-decreasing: if $\hat{f}(n)=f(n)$ then $\hat{f}(n)\leq \hat{f}(n+1)$ since $f$ is non-decreasing, and otherwise, $\hat{f}(n) =\lceil \sqrt{\sup_{t\geq 1} \frac{f(n+t)}{f(t)}} \rceil \leq \lceil \sqrt{\sup_{t\geq 1} \frac{f(n+1+t)}{f(t)}} \rceil \leq \hat{f}(n+1)$ and again $\hat{f}(n)\leq \hat{f}(n+1)$.

Fix $n,m$. If $\hat{f}(n+m)=f(n+m)$ then 
\[
\hat{f}(n+m)=f(n+m)\overset{(*)}{\leq} \hat{f}(n)^2 f(m)\leq \hat{f}(n)^2 \hat{f}(m),
\]
where $(*)$ follows from the maximum definition of $\hat{f}$. Otherwise, fix any $t\geq 1$ and observe that 
\[
\frac{f(n+m+t)}{f(t)}
= \frac{f(n+m+t)}{f(m+t)}\cdot \frac{f(m+t)}{f(t)} \leq \hat{f}(n)^2 \hat{f}(m)^2.
\]
Since in this case $\hat{f}(n+m) = \lceil \sqrt{\sup_{t\geq 1} \frac{f(n+m+t)}{f(t)}} \rceil$, we have $\hat{f}(n+m)^2 \leq \hat{f}(m)^2 \hat{f}(n)^2$, so $\hat{f}(n+m)\leq \hat{f}(n) \hat{f}(m) \leq \hat{f}(n)^2 \hat{f}(m)$. Replacing $f$ by $\hat{f}$,
\begin{equation} \label{eq:Jordan5}
f(n+m) \leq f(n)^2 f(m).
\end{equation}

\medskip

\noindent \emph{Step II: submultiplicativization.} 
We now aim to further replace $f$ defined above by a submultiplicative function. Following \cite{GoswamiPales}, let $g\colon \mathbb{N}\rightarrow \mathbb{N}$ be
\[
g(N) := \min \{f(n_1)\cdots f(n_k)\ :\ n_1+\dots+n_k\geq N\}.
\]
(The construction in \cite{GoswamiPales} is given in terms of sub-additive functions and is equivalent up to applying logarithms.)
Observe that $g$ is well-defined and non-decreasing by definition. Moreover, $g$ is submultiplicative by design: let $N,M$ be given, and fix `minimizing coverings,' that is, $n_1,\dots,n_k$ and $m_1,\dots,m_l$ such that $\sum_{i=1}^{k} n_i\geq N,\sum_{i=1}^{l} m_i \geq M$, and 
\[
g(N)=\prod_{i=1}^{k} f(n_i),\ g(M)=\prod_{i=1}^{l} f(m_i).
\] 
Then 
\[ 
n_1+\dots+n_k+m_1+\dots+m_l\geq N+M,
\]
so 
\[
g(N+M)\leq f(n_1)\dots f(n_k) f(m_1)\dots f(m_l) = g(N) g(M).
\]

We now show that $f\sim g$. Taking $n_1=N$, we evidently see that $g(N)\leq f(N)$ for all $N\in \mathbb{N}$. Conversely, we claim that $f(N) \leq g(3N)$. Fix a minimizing covering $n_1 \geq n_2\geq\dots \geq n_k$ such that $\sum_{i=1}^{k} n_i\geq 3N$ and $g(3N) = f(n_1)\cdots f(n_k)$. First, if $n_1\geq N$ then evidently $g(3N)\geq f(n_1)\geq f(N)$, so let us assume otherwise. Let $t$ be the maximum even index not exceeding $k$. Notice that $n_2+n_4+\dots+n_t\geq N$; otherwise, $n_3+n_5+\dots+n_{t+1}\leq n_2+n_4+\dots+n_t <N$ (if $t=k$, take $n_{t+1}=0$), so $\sum_{i=1}^{k} n_i < n_1 +2N <3N$, a contradiction. Therefore, 
\begin{equation} \label{eq:Jordan6}
n_2+n_4+\dots+n_t\geq N
\end{equation} and thus also 
\[ n_1+n_3+\dots+n_{t-1}\geq N \] 
and by the definition of $g$, it follows that 
\begin{align} \label{eq:Jordan7}
g(3N) & = f(n_1)\cdots f(n_k) \\ & \geq f(n_1)f(n_3)\cdots f(n_{t-1}) \cdot f(n_2)f(n_4)\cdots f(n_t) \nonumber \\ & \geq f(n_2)^2f(n_4)^2\cdots f(n_t)^2 \nonumber
\end{align}
but by \eqref{eq:Jordan5}, 
\begin{eqnarray*}
f(n_2+n_4+\cdots+n_t) & \leq & f(n_2)^2f(n_4+\cdots+n_t) \\ & \vdots & \\ & \leq & f(n_2)^2 f(n_4)^2 \cdots f(n_{t-2})^2 f(n_t)
\end{eqnarray*}
which, together with \eqref{eq:Jordan6} and \eqref{eq:Jordan7}, yields
\[
f(N) \leq f(n_2+n_4+\cdots+n_t) \leq g(3N),
\]
and it follows that $f\sim g$.

Further, by \cite{MZ96}, either $\gamma_J(n)\succeq n^2$ or $\gamma_J(n)=O(n)$, in which case, either $\gamma_J(n)\sim n$ or $\gamma_J(n)=O(1)$. (The stated result in \cite{MZ96} uses the language of the GK-dimension, but the paper actually proves the result as stated above.) 

We henceforth assume that $\gamma_J(n)\succeq n^2$, and thus $g(n)\succeq n^2$.

\medskip

\noindent \emph{Step III: regularization toward the Bell--Zelmanov condition.} Recall that Bell and Zelmanov \cite{BZ} identified a condition satisfied by every growth function of a finitely generated associative algebra (except for bounded and linear ones), and which conversely, every function satisfying it, is asymptotically equivalent to the growth function of some finitely generated associative algebra. Specifically, they showed that if $f\colon \mathbb{N}\rightarrow \mathbb{N}$ satisfies $f'(n)\geq n+1$ and, for all $m\in [n,2n]$, $f'(m)\leq f'(n)^2$, then $f(n)\sim \gamma_A(n)$ for some finitely generated associative algebra $A$.

Let $g\colon \mathbb{N}\rightarrow \mathbb{N}$ be an increasing, submultiplicative function, $g(n)\succeq n^2$. Therefore $g(n)\sim g(n)+2n^2$ (which is submultiplicative too) and thus, without loss of generality, we may assume that $g'(n)\geq 4n-2\geq 2n$. Consider the `dyadic derivative' 
\[
d_k:=\frac{g(2^{k+1}) - g(2^k)}{2^k}
\]
for every $k\geq 0$. 
Notice that
\begin{eqnarray} \label{eq:Jordan8}
d_k & = & \frac{g'(2^{k+1})+\cdots+g'(2^k+1)}{2^k} \\ & \geq & \frac{2\cdot 2^{k+1}+\cdots+2\cdot (2^k+1)}{2^k} \nonumber \\ & \geq & \frac{2\cdot 2^k\cdot 2^k}{2^k}=2^{k+1}. \nonumber
\end{eqnarray}
Let $\beta_k := \sup \{\sqrt[2^i]{d_{k+i}}\ :\ i=0,1,2,\dots\}$ and notice that $\beta_k < \infty$, since $g(2^{k+i+1})^{1/2^i}$ is uniformly bounded over all $i\geq 0$, as $g$ is exponentially bounded. 
We let 
$h(n)=\sum_{i=1}^{n} \lceil \beta_{\lfloor \log_2 i\rfloor } \rceil$ and observe that $h'(n)=\lceil \beta_{\lfloor \log_2 n\rfloor } \rceil$. In other words, $h'(n)$ is the ceiling of $\beta_k$ where $k$ is such that $2^k \leq n < 2^{k+1}$.
This construction can be interpreted as a discrete version of a measure-theoretic concept of the least non-increasing majorant \cite[Pages~207--208]{SinnamonMonotonicity}, given by $f^{\downarrow}(x)\ =$ the essential supremum of $f(y),\ y\geq x$. The discrete version would be $f^{\downarrow}(k) = \sup_{j\geq k} f(j)$. 
The sequence $\{\beta_k\}_{k=0}^{\infty}$ is obtained by considering the sequence $\{\sqrt[2^k]{d_k}\}_{k=0}^{\infty}$, shifting to its least non-increasing majorant as discussed above, and then raising back to the power of $2^k$; in other words, this is a `conjugation' of the discrete majorant by incrementing powers of $2$. The role of these powers of $2$ is to ensure that $\beta_{k+1}\leq \beta_k^2$ as we will see below, which ensures the submultiplicativity-type condition in the Bell--Zelmanov theorem.

First, observe that $h$ satisfies the Bell--Zelmanov condition mentioned above. Indeed, given $n$ and $k$ such that $2^k\leq n < 2^{k+1}$,
\[
h'(n)=\lceil \beta_k \rceil \geq d_k \overset{\eqref{eq:Jordan8}}{\geq} 2^{k+1} \geq n+1.
\]
If $m\in [n,2n]$ then $m < 2^{k+2}$, so $h'(m)=\lceil \beta_k \rceil$ or $h'(m)=\lceil \beta_{k+1} \rceil$. Since 
\[
\sqrt[2^i]{d_{(k+1)+i}} = \left(\sqrt[2^{i+1}]{d_{k+i+1}}\right)^2,
\]
we have that $\beta_{k+1}\leq \beta_k^2$ and hence $\lceil \beta_{k+1} \rceil \leq \lceil \beta_k^2 \rceil \leq \lceil \beta_k \rceil ^ 2$. In all cases, $h'(m)\leq h'(n)^2$, and the Bell--Zelmanov condition is satisfied.

Next, we claim that $h(n)\sim g(n)$.
We have
\begin{eqnarray} \label{eq:Jordan9}
h(2^{k+1}-1) & = & h'(1)+h'(2)+\dots+h'(2^{k+1}-1) \nonumber \\ & \geq & 
\beta_{\lfloor \log_2 1 \rfloor}+\beta_{\lfloor \log_2 2 \rfloor}+\dots+\beta_{\lfloor \log_2 (2^{k+1}-1) \rfloor} \nonumber \\
& = & \beta_0 + 2  \beta_1  + \dots + 2^k  \beta_k  \\ & \geq & d_0+2d_1+\dots+2^k d_k \nonumber \\ & = & (g(2)-g(1))+(g(4)-g(2))+\dots+\left(g(2^{k+1})-g(2^k)\right)=g(2^{k+1})-g(1) \nonumber
\end{eqnarray}
and it follows from Fact \ref{fact} that $g(n)\preceq h(n)$.

Conversely, for every $0\leq m\leq k$, $\beta_m=\sup_{i\geq 0} \{d_{m+i}^{2^{-i}}\}$; so, for some $i=i(m)$, we have $\beta_m \leq d_{m+i}^{2^{-i}}+1$. 

If $i(m)=0$ then $\beta_m \leq d_m+1$ so $2^m \beta_m\leq g(2^{m+1})-g(2^m)+2^m$. Hence 
\begin{equation} \label{eq:Jordan10}
    \sum_{\substack{0\leq m\leq k \\ i(m)=0}} 2^m \beta_m \leq \sum_{m=0}^{k} g(2^{m+1})-g(2^m)+2^m \leq g(2^{k+1}) + 2^{k+1}.
\end{equation}

If $i(m)=1$, $\beta_m \leq \sqrt{d_{m+1}} + 1\leq \frac{\sqrt{g(2^{m+2})}}{2^{m/2}}+1$, so 
\begin{eqnarray} \label{eq:Jordan11}
    \sum_{\substack{0\leq m\leq k \\ i(m)=1}} 2^m \beta_m & \leq & \sum_{m=0}^{k} 2^{m/2} \sqrt{g(2^{m+2})} + 2^m \\ & \leq & (k+1) 2^{k/2} \sqrt{g(2^{k+2})} + 2^{k+1} \nonumber \\ & \leq & 2^{k+1} (\sqrt{g(2^{k+2})}+1) \leq  g(2^{k+2}) \nonumber
\end{eqnarray}
where the last inequality holds since we assume that $g(x)\geq x^2$.

If $i=i(m)\geq 2$, $g(2^{m+i+1})\leq g(2^{m+3})^{2^{i-2}}$ by submultiplicativity, so \[g(2^{m+i+1})^{2^{-i}}\leq g(2^{m+3})^{1/4}\] and \[ \beta_m \leq d_{m+i}^{2^{-i}} + 1\leq g(2^{m+i+1})^{2^{-i}} + 1 \leq g(2^{m+3})^{1/4} + 1. \]
Again since $g(x)\geq x^2$, 
\[
g(2^{k+3})^{3/4} \geq (2^{2(k+3)})^{3/4} \geq 2^{\frac{3}{2}k+1} \geq (k+1)2^k
\]
and therefore
\begin{equation} \label{eq:Jordan12}
\sum_{\substack{0\leq m\leq k \\ i(m) \geq 2}} 2^m \beta_m \leq \sum_{m=0}^{k} 2^m g(2^{m+3})^{1/4} + 2^m \leq (k+1) 2^k g(2^{k+3})^{1/4} + 2^{k+1} \leq g(2^{k+3}) + 2^{k+1}.
\end{equation}

Combining \eqref{eq:Jordan10},\eqref{eq:Jordan11},\eqref{eq:Jordan12}, we get 
\[
\sum_{m=0}^{k} 2^m \beta_m \leq g(2^{k+1})+g(2^{k+2})+g(2^{k+3}) + 2^{k+2} \leq 4g(2^{k+3})
\]
and by \eqref{eq:Jordan9},
\[
h(2^{k+1}-1) = \sum_{m=0}^{k} 2^m \lceil \beta_m \rceil \leq 2 \sum_{m=0}^{k} 2^m \beta_m \leq 8g(2^{k+3})
\]
By Fact \ref{fact}, $h(n)\preceq g(n)$.

\medskip

We conclude that $\gamma_J(n)\sim g(n)\sim h(n)$ and by the Bell--Zelmanov theorem \cite{BZ}, $h(n)\sim \gamma_A(n)$ for some finitely generated associative algebra $A$.

Conversely, if $A$ is a finitely generated associative algebra then $\gamma_A(n)\sim \gamma_{A^{(+)}}(n)$ by \cite{Martinez}.

The proof is complete.
\end{proof}

Notice that an immediate consequence of Theorem \ref{thm:associative} is:

\begin{cor} \label{cor:Jordan}
Let $f\colon \mathbb{N}\rightarrow \mathbb{N}$ be an increasing, polynomially bounded function, such that $f(n)\succeq n^2$. Then there exists a finitely generated Jordan algebra $J$ such that $\gamma_J(n)\sim f(n)$.
\end{cor}

\begin{proof}
By Theorem \ref{thm:associative}, there exists a finitely generated associative algebra $A$ such that $\gamma_A(n)\sim f(n)$. The Jordan algebra $A^{(+)}=(A,\circ),\ a\circ b:=\frac{1}{2}(ab+ba)$ is finitely generated as well \cite{Herstein}, and by \cite[Theorem~2]{MZ22}, $\gamma_{A^{(+)}}(n)\sim \gamma_A(n)\sim f(n)$.
\end{proof}

Conversely, Mart\'inez and Zelmanov \cite{MZ96} proved that the growth functions of Jordan algebras satisfy an analog of Bergman's gap theorem (namely, they cannot be super-linear and sub-quadratic).

\medskip

A related notion is that of \emph{Jordan superalgebras}. Petrogradsky and Shestakov \cite[Corollary~3.2]{PetrogradskyShestakov} proved that given any Lie (super)algebra, there exists a Jordan superalgebra with an equivalent growth function. Hence, by Theorem \ref{thm:Lie},
\begin{cor}
    Let $f\colon \mathbb{N}\rightarrow \mathbb{N}$ be an increasing, polynomially bounded function. Then there exists a finitely generated Jordan superalgebra $\tilde{J}$ such that $\gamma_{\tilde{J}}(n)\sim f(n)$.
\end{cor}

\subsection*{Between two notions of asymptotic equivalence}
We conclude with a short discussion on two similar, but different, notions of asymptotic equivalence relevant to the growth functions of algebraic structures.
Throughout, we used the following notion of asymptotic equivalence of non-decreasing functions:
\[
f\sim g \iff \exists C,D>0:\ \forall n,\ f(n)\leq Cg(Dn)\ \text{and}\ g(n)\leq Cf(Dn).
\]
However, one may consider the finer notion
\[
f\approx g \iff \exists C>0:\ \forall n \gg 1,\ f(n)\leq g(Cn)\ \text{and}\ g(n)\leq f(Cn).
\]
Given a group, semigroup, or algebra (associative, Lie, Jordan, etc.), its growth functions with respect to different generating subsets (or subspaces) may numerically differ, but they are all $\approx$-equivalent to each other. Nevertheless, $\sim$ has some natural advantages over $\approx$; for instance, it identifies all bounded functions, and, in general, makes the growth functions of algebraic structures invariant under `finite extensions': if $I\triangleleft A$ is a finite-dimensional ideal, then $\gamma_A(n)\sim \gamma_{A/I}(n)$, and if $A\subseteq B$ are associative algebras such that $B$ is a finitely generated $A$-module, then $\gamma_A(n)\sim \gamma_B(n)$.

In \cite{BZ}, Bell and Zelmanov obtained a classification of the growth functions of associative algebras and semigroups up to the finer equivalence relation $\approx$; the characterization in Theorem \ref{thm:Jordan} avoids the discrete derivative submultiplicative condition, but is valid only up to~$\sim$.

To what extent do $\sim$ and $\approx$ really differ from each other? Zelmanov \cite{AAJZ_res,Spa} asked whether every increasing, submultiplicative function that grows at least quadratically (to avoid Bergman's gap theorem), is asymptotically equivalent to the growth functions of some finitely generated associative algebra. This was shown not to be the case up to $\approx$ in \cite{GaApp}; however, as we mentioned, the answer turns out to be affirmative when the equivalence relation is $\sim$.

\end{document}